\documentclass[12pt]{article}

\usepackage[]{amsmath, amssymb, amsthm, amsfonts,enumerate, MnSymbol, relsize, mathtools, url, xcolor}
\usepackage[hidelinks]{hyperref}
\usepackage[utf8]{inputenc}
\usepackage[normalem]{ulem}
\usepackage[margin= 2.5 cm]{geometry}
\theoremstyle{plain}
\newtheorem{theorem}{Theorem}[section]
\newtheorem{claim}[theorem]{Claim}
\newtheorem{lemma}[theorem]{Lemma}
\newtheorem{corollary}[theorem]{Corollary}

\theoremstyle{definition}
\newtheorem{definition}[theorem]{Definition}
\newtheorem*{remark*}{Remark}
\newtheorem*{notation*}{Notation}

\makeatletter
\DeclareRobustCommand{\overleftharpoon}{\mathpalette{\overarrow@\leftharpoonfill@}}
\def\leftharpoonfill@{\arrowfill@\leftharpoondown\mn@relbar\mn@relbar}
\DeclareMathSymbol{\leftharpoondown}{\mathrel}{MnSyA}{'112}
\DeclareMathSymbol{\mn@relbar}{\mathrel}{MnSyA}{'320}
\makeatother

\makeatletter
\newcommand{\case}[2]{%
  \par
  \addvspace{\medskipamount}%
  \noindent\textbf{Case #1: #2}\nobreak
}
\makeatother

\begin{document}

\title{Hypergraphs of arbitrary uniformity with vanishing codegree Tur{\'a}n density}
\author{James Sarkies\footnote{Trinity College, University of Cambridge, United Kingdom. Research supported by the Trinity College Summer Studentship Scheme. Email: \href{mailto:js2756@cam.ac.uk}{\nolinkurl{js2756@cam.ac.uk}}}}
\date{}
\maketitle

\begin{abstract}
    The codegree Tur{\'a}n density \(\pi_{\text{co}}(F)\) of a \(k\)-uniform hypergraph (or \(k\)-graph) \(F\) is the infimum over all
    \(d\) such that a copy of \(F\) is contained in any sufficiently large \(n\)-vertex \(k\)-graph \(G\) with the property that any 
    \((k-1)\)-subset of \(V(G)\) is contained in at least \(dn\) edges. The problem of determining \(\pi_{\text{co}}(F)\) for a \(k\)-graph \(F\) is 
    in general very difficult when \(k \geq 3\), and there were previously very few nontrivial examples of \(k\)-graphs \(F\) for which \(\pi_{\text{co}}(F)\) was 
    known when \(k \geq 4\).

    In this paper, we prove that \(C_\ell^{(k)-}\), the \(k\)-uniform tight cycle of length \(\ell\) minus an edge, has vanishing codegree Tur{\'a}n density if and only if \(\ell \equiv 0, \pm 1  \text{ ({mod} } k\text{)}\) when \(\ell \geq k + 2\). This generalises a result of Piga, Sales and Sch{\"u}lke, who proved that \(\pi_\text{co}(C_\ell^{(3)-}) = 0\) when \(\ell \geq 5\).  The method used to prove that \(\pi_\text{co}(C_\ell^{(k)-}) = 0\) when \(\ell \equiv \pm 1  \text{ ({mod} } k\text{)}\) and \(\ell \geq 2k - 1\) in fact gives a rather larger class of \(k\)-graphs with vanishing codegree Tur{\'a}n density.
     We also answer a question of Piga and Sch{\"u}lke by proving that another
    family of \(k\)-graphs, studied by them, has vanishing codegree Tur{\'a}n density.
\end{abstract}
 
\section{Introduction}
A \(k\)-uniform hypergraph (or \(k\)-graph) \(F\) is a set \(V(F)\), called the vertices of \(F\), together with a set \(E(F)\) of \(k\)-subsets of \(V(F)\),
called the edges of \(F\). We shall often denote the edge \(\{v_1, \ldots, v_k\}\) simply by \(v_1 \ldots v_k\).
Given a \(k\)-graph \(F\), one of the central problems in extremal combinatorics is 
to determine the Tur{\'a}n number \(\text{ex}(n, F)\), defined to
be the largest possible number of edges in an \(n\)-vertex \(k\)-graph containing no copy of \(F\) as a subgraph.
Since we are interested in the limiting behaviour, we may define the Tur{\'a}n density of \(F\) as

\begin{equation*}
    \pi(F) = \lim_{n \to \infty} \frac{\text{ex}(n, F)}{\binom{n}{k}}.
\end{equation*}

A simple averaging argument shows that this limit always exists. The problem of determining \(\pi(F)\) for a \(k\)-graph \(F\) is answered by the Erd{\H{o}}s-Stone-Simonovits theorem \cite{ erdos1946structure, erdos1966limit} when \(k = 2\); however the problem becomes notoriously difficult
when \(k \geq 3\). Indeed, even the Tur{\'a}n density \(\pi(K_4^{(3)})\), where \(K_4^{(3)}\) is the complete \(3\)-graph on four vertices, is still unknown more than eighty years after the question was first asked by Tur{\'a}n \cite{turan1941extremal}. For a 
survey of known results regarding the Tur{\'a}n density, see for example the survey of Keevash \cite{keevash2011hypergraph}.

In this paper, we consider a natural variant of the Tur{\'a}n density called the codegree Tur{\'a}n density, introduced by Mubayi and Zhao \cite{mubayi2007co}. Given a \(k\)-graph \(G\) and a subset \(S\) of \(V(G)\), 
the degree of \(S\), denoted by \(d_G(S)\), is the number of edges of \(G\) containing \(S\). The minimum codegree \(\delta_{k-1}(G)\) of \(G\) is defined to be the minimum of \(d_{G}(S)\) over all \((k-1)\)-subsets \(S\) of \(V(G)\).

Given a \(k\)-graph \(F\), the codegree Tur{\'a}n number \(\text{ex}_\text{co}(n, F)\) is defined to be the largest possible value of \(\delta_{k-1}(G)\) when \(G\) is an \(n\)-vertex \(k\)-graph containing no copy of \(F\). Again focusing on the limiting behaviour, we 
define the codegree Tur{\'a}n density of \(F\) as 

\begin{equation*}
    \pi_\text{co}(F) = \lim_{n \to \infty} \frac{\text{ex}_\text{co}(n, F)}{n}.
\end{equation*}

This limit always exists \cite{mubayi2007co}, and it is not hard to see that \(\pi_\text{co}(F) \leq \pi(F)\). Similarly to 
the usual Tur{\'a}n density, the problem of determining \(\pi_{\text{co}}(F)\) for a \(k\)-graph \(F\) seems to be very difficult
in general when \(k \geq 3\). 

Some recent progress was made in the \(k = 3\) case, when Ding, Lamaison, Liu, Wang and Yang \cite{ding20243} proved that \(\pi_{\text{co}}(F) = 0\) whenever \(F\) is a so-called layered \(3\)-graph with vanishing uniform Tur{\'a}n density, where the uniform Tur{\'a}n density is another variant of the Tur{\'a}n density introduced by Erd{\H{o}}s and S{\'o}s \cite{erdHos1982ramsey}. There are a small number more of nontrivial examples of \(k\)-graphs whose codegree Tur{\'a}n density is known (most of them have \(k = 3\)). Among them are  the \(3\)-graph on four vertices with three edges, \(K_4^{(3)-}\), due to Falgas-Ravry, Marchant, Vaughan and Volec \cite{falgas2017codegree}; the \(3\)-graph \(F_{3, 2}\) with \(V(F_{3, 2}) = \{1, 2, 3, 4, 5\}\) and \(E(F_{3, 2}) = \{123, 124, 125, 345\}\) due to Falgas-Ravry, Marchant, Pikhurko and Vaughan \cite{falgas2015codegree}; and some of the projective geometries \(PG_{m}(q)\) when \(m = 2, 3\) due to Keevash and Zhao \cite{keevash2007codegree} and Zhang and Ge \cite{zhang2021codegree}.

In this paper, we determine the codegree Tur{\'a}n densities of a number of \(k\)-graphs when \(k \geq 3\). Our first result is concerned with the following family of hypergraphs.

\begin{definition}
The \textit{\(k\)-uniform tight cycle of length \(\ell\)}, denoted by \(C_\ell^{(k)}\), is the \(k\)-graph with \(V(C_\ell^{(k)}) = \{v_1, \ldots, v_\ell\}\) and \(E(C_\ell^{(k)}) = \{v_i v_{i + 1} \ldots v_{i + k - 1} : 1 \leq i \leq \ell\}\), where the  indices are taken modulo \(\ell\).
\end{definition}

 Let \(C_\ell^{(k)-}\) be the \(k\)-graph \(C_\ell^{(k)}\) minus an edge. Determining the Tur{\'a}n densities (and variants thereof) for \(C_\ell^{(k)}\) and \(C_\ell^{(k)-}\) has received some recent interest. 
 The most recent progress was made by Sankar \cite{sankar2024turandensity4uniformtight}, who proved that \(\pi(C_\ell^{(4)}) = \frac{1}{2}\) when \(\ell\) is sufficiently large and not divisible by \(4\). 
 
 The other known results concern the \(k = 3\) case. If \(3 \mid \ell\), then \(C_\ell^{(3)}\) is tripartite, so \(\pi(C_\ell^{(3)}) = 0\) due to a result of Erd{\H{o}}s \cite{erdos1964extremal} (whence all of \(\pi(C_\ell^{(3)})\), \(\pi_{\text{co}}(C_\ell^{(3)})\), \(\pi(C_\ell^{(3)-})\) and \(\pi_{\text{co}}(C_\ell^{(3)-})\) vanish). When \(3 \nmid \ell\), Lidick{\'y}, Mattes and Pfender \cite{lidicky2024hypergraph} proved that \(\pi(C_\ell^{(3)-}) = \frac{1}{4}\) if \(\ell \geq 5\); Kam{\v{c}}ev, Letzter and Pokrovskiy \cite{kamvcev2024turan} proved that \(\pi(C_\ell^{(3)}) = 2\sqrt{3} - 3\) whenever \(\ell\) is sufficiently large; and Piga, Sanhueza-Matamala and Schacht \cite{piga2024codegree} and Ma \cite{ma2024codegree} proved that \(\pi_{\text{co}}(C_\ell^{(3)}) = \frac{1}{3}\) for \(\ell \geq 10\).

 Piga, Sales and Sch{\"u}lke \cite{piga2023codegree} proved that \(\pi_{\text{co}}(C_\ell^{(3)-}) = 0\) for all \(\ell \geq 5\). We prove the following result for higher uniformities.

\begin{theorem} \label{thm:tightCycle}
    Let \(k \geq 3\), \(\ell \geq k + 2\) be integers. Then \(\pi_{\textnormal{co}}(C_\ell^{(k)-}) = 0\) if and only if \(\ell \equiv 0, \pm 1  \pmod{k}\).
\end{theorem}

We deduce the sufficiency of the condition \(\ell \equiv 0, \pm 1 \pmod{k}\) from the more general Theorem \ref{thm:sufficientCondition}, which gives a sufficient condition for a \(k\)-graph \(F\) to satisfy \(\pi_\text{co}(F) = 0\). In order to state Theorem \ref{thm:sufficientCondition}, we make the following definitions.

\begin{definition}
    An \textit{ordered \(k\)-graph} is a \(k\)-graph \(G\), together with a total ordering \(<\) on \(V(G)\). If \(F\), \(H\) are ordered \(k\)-graphs, then we say that \(F\) is an \textit{ordered subgraph} of \(H\) if there is an (injective) order-preserving map 
    \(f: V(F) \to V(H)\) such that if \(v_1 \ldots v_k \in E(F)\), then \(f(v_1) \ldots f(v_k) \in E(H)\).

    Let \(V\) be a totally ordered set of size \(kt\), and let \(V_1\) be set of the \(t\) smallest elements of \(V\); \(V_2\) be the set of the next \(t\) smallest; and so on. The \textit{ordered complete \(k\)-partite \(k\)-graph} \(K_{t < \ldots < t}^{(k)}\) is the ordered \(k\)-graph with vertex set \(V\) and \(E(K_{t < \ldots < t}^{(k)}) = \{v_1 \ldots v_k : v_1 \in V_1, \ldots, v_k \in V_k\}\).
    \end{definition}

    \begin{definition}   
        Given a \(k\)-graph \(G\) and a vertex \(v \in G\), the \textit{link graph \(L_G(v)\)} is the \((k - 1)\)-graph with vertices \(V(L_{G}(v)) = V(G) \setminus \{v\}\), and
        \(w_1 \ldots w_{k - 1} \in E(L_{G}(v))\) if and only if \(v w_1 \ldots w_{k - 1} \in E(G)\). 
        
        Also, we write \(L_G^+(v)\) for the \((k-1)\)-graph obtained by adding the vertex \(v\) (and no edges) to
        \(L_G(v)\). (Hence, in particular, \(V(L_G^+(v)) = V(G)\).)
        
        When the \(k\)-graph \(G\) is clear from the context, we may omit it from the notation and simply write \(L(v)\) for \(L_G(v)\).
    \end{definition}

We can now state our sufficient condition for vanishing codegree Tur{\'a}n density.

    \begin{theorem} \label{thm:sufficientCondition}
    Suppose \(F\) is a \(k\)-graph with the following properties:
    \begin{enumerate}[(i)]
        \item There is a partition \(V(F) = V_1 \dot\cup V_2\) such that for any \(e \in E(F)\), \(|e \cap V_1| = 1\); and
        \item There is a total ordering of \(V_2\) and an integer \(t\) such that for any \(v \in V_1\), the link graph \(L_F(v) [V_2]\) is 
    an ordered subgraph of \(K_{t<\ldots <t}^{(k - 1)}\).
    \end{enumerate}
     
    Then \(\pi_\textnormal{co}(F) = 0\).
    \end{theorem}

Our final result gives another family of hypergraphs with vanishing codegree Tur{\'a}n density.
\begin{definition}
    The \textit{\(k\)-uniform zycle of length \(\ell\)}, denoted by \(Z_\ell^{(k)}\), is the \(k\)-graph with
    \begin{equation*}
        V(Z_\ell^{(k)}) = \{v^j_i : 1 \leq i \leq \ell, 1 \leq j \leq k - 1\}
    \end{equation*}
    \begin{equation*}
        E(Z_\ell^{(k)}) = \{v^1_i v^2_i \ldots v^{k - 1}_i v^j_{i + 1} : 1 \leq i \leq \ell, 1 \leq j \leq k - 1\},
    \end{equation*}
    where the lower indices are taken modulo \(\ell\).
\end{definition}

We write \(Z_\ell^{(k)-}\) to denote the \(k\)-graph \(Z_\ell^{(k)}\) minus one edge. Piga and Sch{\"u}lke \cite{piga2023hypergraphs} proved that \(\pi_{\text{co}}(Z_\ell^{(3)-}) = 0\) 
for \(\ell \geq 3\), and asked if \(\pi_{\text{co}}(Z_{\ell}^{(k)-}) = 0\) for \(k \geq 4\) and \(\ell\) sufficiently large. We answer this question affirmatively, even without the assumption that \(\ell\) is large.

\begin{theorem} \label{thm:zycle}
    For all \(k \geq 3\) and \(\ell \geq 3\), \(\pi_{\textnormal{co}}(Z_\ell^{(k)-}) = 0\).
\end{theorem}

The rest of the paper is organised as follows. In the remainder of this section, we specify our notation and state a result of Mubayi and Zhao which will be useful to us (Lemma \ref{lem:supersaturation}). In Section 2, we prove Theorem \ref{thm:sufficientCondition}. In Section 3, we apply Theorem \ref{thm:sufficientCondition} to prove one direction of Theorem \ref{thm:tightCycle}, and provide constructions which prove the other direction. In Section 4, we prove Theorem \ref{thm:zycle}. 

\subsection{Notation and a Useful Result}

    Given a \(k\)-graph \(G = (V, E)\) and \(e \in V^{(k-1)}\), we define the \textit{neighbourhood} of \(e\) as 
    \begin{equation*}
        N(e) = \bigl\{v \in V : e \cup \{v\} \in E\bigr\}.
    \end{equation*}
    Given a subset \(S\) of \(V\), we define the \textit{back neighbourhood} of \(S\) as
    \begin{equation*}
        \overleftharpoon{N}(S) = \bigl\{e \in V^{(k-1)} : e \cup \{v\} \in E \text{ for every } v \in S\bigr\}.
    \end{equation*}
    
    If \(f \in V^{(k-2)}\) and \(v \in V \setminus f\), then we write \(fv\) to denote \(f \cup \{v\} \in V^{(k-1)}\). If \(e_1, e_2 \in V^{(k-1)}\), we write \(e_1 \rhd e_2\) if \(e_2 \in N(e_1)^{(k - 1)}\).  Thus
    a zycle \(Z_\ell^{(k)}\) is a sequence of pairwise disjoint \((k-1)\)-sets \(e_1, \ldots, e_\ell\) such that
    \(e_1 \rhd \cdots \rhd e_\ell \rhd e_1\).

Given a \(k\)-graph \(F\), define \(\text{ex}_{\text{co}}^\text{hom}(n, F)\) to be the largest possible value of \(\delta_{k-1}(G)\) when \(G\) is an \(n\)-vertex \(k\)-graph such that there is no homomorphism \(F \to G\). Similarly to before, we define 
\begin{equation*}
    \pi_{\text{co}}^{\text{hom}}(F) = \lim_{n \to \infty} \frac{\text{ex}_\text{co}^{\text{hom}}(n, F)}{n}.
\end{equation*}

 Using the supersaturation phenomenon for codegree Tur{\'a}n density of Mubayi and Zhao \cite{mubayi2007co}, it can be shown that this limit exists, and moreover that \(\pi_{\textnormal{co}}(F)  = \pi_{\textnormal{co}}^{\textnormal{hom}}(F)\) for any \(k\)-graph \(F\).

\begin{lemma} \label{lem:supersaturation}
    For any \(k\)-graph \(F\), \(\pi_{\textnormal{co}}(F)  = \pi_{\textnormal{co}}^{\textnormal{hom}}(F)\).
\end{lemma}

\section{A Sufficient Condition for Zero Codegree Density}
In this short section we prove Theorem \ref{thm:sufficientCondition}, the approach for which is based on an idea in \cite{ding20243}. 

\begin{proof}[Proof of Theorem \ref{thm:sufficientCondition}]
    Fix \(\varepsilon > 0\), and suppose that \(n\) is sufficiently large in terms of \(\varepsilon\). Let \(G = (V, E)\) be a \(k\)-graph with \(|V| = n\) and \(\delta_{k-1}(G) \geq \varepsilon n\).
    By Lemma \ref{lem:supersaturation}, it suffices to find a homomorphism \(F \to G\).
    
    Write \(|V_1| = a\), \(|V_2| = b\), and write \(V_1 = \{w_1, \ldots, w_a\}\). Impose an arbitrary total ordering \(<\) on the vertex set \(V(G)\), thereby making \(G\) an ordered \(k\)-graph.
    
    We assign colours to edges of the complete \(b\)-graph on vertex set \(V(G)\) as follows. The edge \(U \in V(G)^{(b)}\) can be coloured 
    \(i\) if there is no vertex \(x \in V(G)\) such that \(L_G^+(x)[U]\)  contains \(L_F(w_i)[V_2]\) as an ordered subgraph. If an edge can be assigned multiple colours, we pick one arbitrarily.
     
    We assume for sake of contradiction that every edge is assigned a colour. Let \(n_0\) be a large integer divisible by \(k - 1\), to be chosen later. If \(n\) is large enough, then by Ramsey's theorem there is a monochromatic clique of size 
    \(n_0\), say of colour \(i\) and on vertices \(V'\). 
    
    Let \(s\) be the integer such that \((k - 1)s = n_0\). Let \(H\) be the copy of \(K_{s < \ldots < s}^{(k - 1)}\) 
    on the vertex set \(V'\). Note that
    \begin{equation*}
        \sum_{x \in V(G)} e(L_{G}^+(x)[V'] \cap H) = \sum_{y_1 \ldots y_{k -1} \in E(H)} d_G(y_1 \ldots y_{k - 1}) \geq e(H) \cdot \varepsilon n.
    \end{equation*}
    By averaging, there exists a vertex \(x \in V(G)\) such that the edge density of \(L_G^+(x)[V'] \cap H\) is at least \(\varepsilon \cdot e(H) / \binom{|V'|}{k - 1} > \frac{\varepsilon (k - 1)!}{(k - 1)^{k - 1}}\). So 
    if \(n_0\) is chosen sufficiently large, then \(L_{G}^+(x)[V'] \cap H\) has a subgraph \(K_{t, \ldots, t}\), since \(\pi(K_{t, \ldots, t}) = 0\) due to a result of Erd{\H{o}}s \cite{erdos1964extremal}. By the 
    definition of \(H\), this copy of \(K_{t, \ldots, t}\) is also a copy of the ordered \((k - 1)\)-graph \(K_{t < \ldots < t}^{(k - 1)}\) in \(L_G^+(x)[V']\).

    But this is a contradiction, as it was assumed that
    \(L_F(w_i)[V_2]\) is an ordered subgraph of \(K_{t<\ldots<t}^{(k - 1)}\), and so the vertex set of this copy of \(L_F(w_i)[V_2]\) could not have received the colour \(i\).

    So there is an edge which is assigned no colour - let \(U\) be such an edge. Then for each \(1 \leq i \leq a\), there is a vertex \(x_i \in V(G)\) such that 
    \(L_F(w_i)[V_2]\) is an ordered subgraph of \(L_G^+(x_i)[U]\). Hence there is a homomorphism \(F \to G\) - namely, the map that sends \(w_i \mapsto x_i\) and 
    sends \(V_2\) to \(U\) while preserving order.
\end{proof}

\section{The Codegree Density of \texorpdfstring{\(C_\ell^{(k)-}\)}{Tight Cycles minus an Edge}}

In this section, we prove Theorem \ref{thm:tightCycle}. We split this section into the proof that the condition \(\ell \equiv 0, \pm 1 \pmod{k}\) is sufficient, and the proof that it is necessary.

\subsection{\texorpdfstring{\(\ell \equiv 0, \pm 1 \pmod{k}\)}{}}
In this subsection we use Theorem \ref{thm:sufficientCondition} to prove that if \(\ell \equiv 0, \pm 1 \pmod{k}\) and \(\ell \geq k + 2\), then \(\pi_\text{co}(C_\ell^{(k)-}) = 0\).

\begin{proof}[Proof of Theorem \ref{thm:tightCycle} (\(\impliedby\))]
    We firstly note that if \(\ell \equiv 0 \pmod{k}\), then \(C_\ell^{(k)-}\) is \(k\)-partite. A result of Erd{\H{o}}s \cite{erdos1964extremal} then says that \(\pi(C_\ell^{(k)-}) = 0\), and so \(\pi_{\text{co}}(C_\ell^{(k)-})  = 0\).

    Also, if \(\ell \geq 3k - 1\), there is a homomorphism \(C_\ell^{(k)-} \to C_{\ell - k}^{(k)-}\). Indeed, if the vertices of \(C_{\ell - k}^{(k)-}\) are \(v_1, \ldots, v_{\ell - k}\) (in that order) with the missing edge being \(v_{\ell - k} v_1 \ldots v_{k - 1}\), then the cyclic ordering \(v_1, v_2, \ldots, v_k, v_1, v_2, \ldots, v_{\ell - k}\) gives a homomorphic copy of \(C_\ell^{(k)-}\) (with the missing edge being \(v_{\ell - k} v_1 \ldots v_{k - 1}\)). So by Lemma \ref{lem:supersaturation} and induction, it suffices to prove that \(\pi_{\text{co}}(C_{2k - 1}^{(k)-}) = 0\) and \(\pi_{\text{co}}(C_{2k + 1}^{(k)-}) = 0\).

    Firstly, we show that \(C_{2k - 1}^{(k)-}\) satisfies the conditions in Theorem \ref{thm:sufficientCondition}. Suppose that the vertices of \(C_{2k - 1}^{(k)-}\) are \(v_1, v_3, \ldots \allowbreak, v_{2k - 3}, w, v_2, v_4, \ldots, v_{2k - 4}, x\) (in that order) with the missing 
edge being \(w v_2 v_4 \ldots v_{2k - 4} x\). Let \(V_1 = \{w, x\}\) and \(V_2 = \{v_1, \ldots, v_{2k - 3}\}\) with the ordering \(v_1 < v_2 < \ldots < v_{2k - 3}\). Each edge of \(C_{2k - 1}^{(k)-}\) 
contains exactly one vertex from \(V_1\). Consider the ordered complete $(k-1)$-partite \((k -1)\)-graph \(K_{2<\ldots < 2}^{(k - 1)}\) on vertices \(u_1 < \ldots < u_{2k - 2}\). 

Then sending \(v_i \mapsto u_i\) makes \(L(w)[V_2]\) an ordered subgraph of \(K_{2 < \ldots < 2}^{(k - 1)}\); and sending \(v_i \mapsto u_{i+1}\) makes \(L(x)[V_2]\) an ordered subgraph of \(K_{2 < \ldots < 2}^{(k - 1)}\).

Now we show that \(C_{2k + 1}^{(k)-}\) has satisfies the conditions in Theorem \ref{thm:sufficientCondition}. Suppose that the vertices of \(C_{2k + 1}^{(k)-}\) are \(v_1, v_3, \ldots, v_{2k - 1}, w, v_2, v_4, \ldots, v_{2k - 2}, x\) (in that order) with the missing 
edge being \(v_1 v_3 \ldots v_{2k - 1}\). Let \(V_1 = \{w, x\}\) and \(V_2 = \{v_1, \ldots, v_{2k - 1}\}\) with the ordering \(v_1 < \ldots < v_{2k - 1}\). Each edge of \(C_{2k + 1}^{(k)-}\)
contains exactly one vertex from \(V_1\). Consider the ordered complete $(k-1)$-partite \((k - 1)\)-graph \(K_{4 < \ldots < 4}^{(k - 1)}\) on vertices \(u_1' < u_1 < u_2 < u_1'' < \ldots < u_{k - 1}' < u_{2k-3} < u_{2k - 2} < u_{k - 1}''\).

Then sending \(v_i \mapsto u_{i - 1}\) for \(2 \leq i \leq 2k - 1\) and \(v_1 \mapsto u_1'\) makes 
\(L(w)[V_2]\) an ordered subgraph of \(K_{4 < \ldots < 4}^{(k - 1)}\); and sending \(v_i \mapsto u_i\) for \(1 \leq i \leq 2k - 2\) and \(v_{2k - 1} \mapsto u_{k - 1}''\) makes \(L(x)[V_2]\) an ordered subgraph of \(K_{4 < \ldots < 4}^{(k - 1)}\).

Hence by Theorem \ref{thm:sufficientCondition}, \(\pi_{\text{co}}(C_{2k - 1}^{(k)-}) = \pi_{\text{co}}(C_{2k + 1}^{(k)-}) = 0\) as required.
\end{proof}

\subsection{\texorpdfstring{\(\ell \nequiv 0, \pm 1 \pmod{k}\)}{}}

Now we prove that if \(\ell \nequiv 0, \pm 1 \pmod{k}\), then \(\pi_\text{co}(C_\ell^{(k)-}) > 0\). To do this, we construct a \(k\)-graph \(G\) of linear minimum codegree such that any copy of \(C_\ell^{(k)-}\) must contain a certain type of `bad' edge. Since there are not too many `bad' edges, we can delete them, and add new edges to restore the linear minimum codegree condition, while avoiding introducing any copies of \(C_\ell^{(k)-}\).

\begin{proof}[Proof of Theorem \ref{thm:tightCycle} (\(\implies\))]
Given integers \(m\) and \(N\), let \(G_{m, N}\) be the \(k\)-graph defined as follows. The vertex set \(V(G_{m, N})\) is the union of pairwise disjoint sets \(V_{1}, \ldots, V_{m}\)  of size \(N\), and 
we define the function \(f : V(G_{m, N}) \to \mathbb{Z}/m\mathbb{Z}\) such that \(v \in V_{f(v)}\). (Note that here and throughout we simply write \(x\) for \(x + m\mathbb{Z}\).) The edges are then given by \(v_1 \ldots v_k \in E(G_{m, N})\) if and only if \(f(v_1) + \ldots + f(v_k) = 1\).

Note that \(|V(G_{m, N})| = mN\) and \(\delta_{k - 1}(G_{m, N}) \geq N - k + 1\).

Suppose that the vertices \(v_1, \ldots, v_\ell\) (in that order) form a copy of \(C_\ell^{(k)-}\) in \(G_{m, N}\), with the missing edge being \(v_\ell v_1 \ldots v_{k - 1}\). Note that if \(v_i \ldots v_{i + k -1}\) and \(v_{i + 1} \ldots v_{i + k}\) (with indices taken modulo \(\ell\)) are both edges in \(G_{m, N}\), then \(f(v_i) = f(v_{i + k})\). Hence we obtain the following equation, with the indices taken modulo \(\ell\):
\begin{equation*} \tag{\dag}
    f(v_i) = f(v_{i + k}), \text{ for } i \nequiv  0, -1 \pmod{\ell}.
\end{equation*}
    
Firstly suppose that \(d = \text{gcd}(k, \ell) > 1\). Write \(h = \ell / d\). Then, if \(-1 \leq i \leq d - 2\), then none of \(i + k, i + 2k, \ldots, i + (h - 1)k\) are congruent to \(-1, 0, \ldots, d - 2 \pmod{\ell}\). In particular, they are not congruent to \(-1, 0 \pmod{\ell}\), so by (\(\dag\)) we obtain

\begin{equation*}
             f(v_{i + k}) = f(v_{i + 2k}) = \cdots = f(v_{i + (h-1)k}) = f(v_i).
    \end{equation*}
    
    Hence if \(j \equiv i \pmod{d}\), then \(f(v_i) = f(v_j)\).

    Put \(m = \frac{k}{d}\). Note that \(k \neq d\) because \(k \nmid \ell\), so \(m > 1\). Then \(f(v_1) + \ldots + f(v_k) = \frac{k}{d} (f(v_1) + \ldots + f(v_d)) \neq 1\), whence \(v_1 \ldots v_k \notin G_{m, N}\). However, \(v_1 \ldots v_k\) was an edge in our putative copy of \(C_\ell^{(k)-}\), a contradiction. So \(C_\ell^{(k)-}\) is not a subgraph of \(G_{m, N}\) for any \(N\), and so \(\pi_{\text{co}}(C_\ell^{(k)-}) \geq 1/m > 0\).
    
    For the remainder of the proof, we suppose that \(\text{gcd}(k, \ell) = 1\).
    
    We say that a set of vertices \(\{w_1, \ldots, w_q\}\) is \textit{of the form \(a_1^{\alpha_1} \ldots a_p^{\alpha_p}\)} if exactly \(\alpha_i\) of \(f(w_1), \ldots, f(w_q)\) are equal to \(a_i\) for each \(1 \leq i \leq p\).

    \renewcommand{\qedsymbol}{\ensuremath{\blacksquare}}

    \begin{claim} \label{claim:tightCycle1}
        There exist integers \(t, s \geq 2\) (independent of \(m\)) such that the following holds. If the vertices \(v_1, \ldots, v_\ell\) form a copy of \(C_\ell^{(k)-}\) in \(G_{m, N}\) (with missing edge \(v_\ell v_1 \ldots v_{k-1}\)), then there exist \(a, b \in \mathbb{Z}/m\mathbb{Z}\) such that the edge \(v_1 \ldots v_k\) is of the form \(a^t b^s\). (Note then that \(t + s = k\) and \(ta + sb = 1\).)
    \end{claim}
    
    \begin{proof}
        Arrange all the residue classes modulo \(\ell\) in a circle, in the clockwise order \([k], [2k], \ldots, \allowbreak  [\ell k]\). Let \(X\) be the set of residue classes \([x]\) such that either \([x] \neq [0], [-1]\) and \([x]\) lies on the clockwise arc from \([-1]\) to \([0]\), or \([x] = [0]\). Let \(Y\) be the complementary set of residue classes: i.e. \([y] \in Y\) if either \([y] \neq [0], [-1]\) and \([y]\) lies on the clockwise arc from \([0]\) to \([-1]\), or \([y] = [-1]\).
        
        Then by (\(\dag\)), there exist \(a, b \in \mathbb{Z}/m\mathbb{Z}\) such that \(f(v_x) = a\) whenever \([x] \in X\), and \(f(v_y) = b\) whenever \([y] \in Y\).

        Set \(t = \big|\{[1], \ldots, [k]\} \cap X\big|\), \(s = \big|\{[1], \ldots, [k]\} \cap Y\big|\). Then \(v_1 \ldots v_k\) is of the form \(a^t b^s\), so it suffices to show that \(t,s \geq 2\). For this, write \(\ell = qk + r\), where \(2 \leq r \leq k -2\). It then immediately follows that \([k-1], [k-1- r] \in X\), and \([k], [k - r] \in Y\), whence \(t, s \geq 2\).
    \end{proof}

    Our strategy to construct a \(C_\ell^{(k)-}\)-free \(k\)-graph with linear minimum codegree will thus be to remove all edges of the form \(a^t b^s\) from \(G_{m, N}\) whenever \(ta + sb = 1\), and add new edges of the form \(a^{t - 1} b^s c(a, b)\) and \(a^t b^{s - 1} d(a, b)\) to restore the linear minimum codegree condition. In order to avoid introducing any copies of \(C_\ell^{(k)-}\), we need to be slightly careful about how we choose \(c(a,b)\) and \(d(a, b)\). The conditions which we demand are satisfied are conditions (i)-(v) in the following claim.

    \begin{claim} \label{claim:tightCycle2}
         We can choose \(m\) such that the following holds. Suppose \(a, b \in \mathbb{Z}/m\mathbb{Z}\) are such that \(ta + sb = 1\). Then we can choose \(c = c(a,b), d = d(a, b) \in \mathbb{Z}/m\mathbb{Z}\) such that if we write \(a' = a'(a, b) = 1 - ((t - 2)a + sb + c)\), \(b' = b'(a, b) = 1 - ((t - 1)a + (s - 1)b + c)\), \(a'' = a''(a, b) = 1 - ((t - 1)a + (s - 1)b +d)\) and \(b'' = b''(a, b) = 1 - (ta + (s - 2)b + d)\), then the following conditions are satisfied:
         \begin{enumerate}[(i)]
        \item \(a, b, c, a', b'\) are pairwise distinct;
        \item \(a, b, d, a'', b''\) are pairwise distinct;
        \item If \(\{x, y\} \in \{a, b, c, a', b'\}^{(2)} \setminus \{\{a', b'\}, \{a, b\}\}\) or \(x = b, y = a\), then \(tx + sy \neq 1\);
        \item If \(\{x, y\} \in \{a, b, d, a'', b''\}^{(2)} \setminus \{\{a'', b''\}, \{a, b\}\}\) or \(x = b, y = a\), then \(tx + sy \neq 1\); and
        \item \(c \neq d\), \(c \neq a''\), \(d \neq b'\).
    \end{enumerate}
    \end{claim}

    \begin{proof}
       We choose \(m = kp\) where \(p\) is a prime with \(p > k\) and such that \(m \geq 52.\)

        Using the facts that \(ta + sb = 1\), \(t + s = k\) and \(k \mid m\), we can see that \(a \neq b\). Hence we also have \(a' \neq b'\). So satisfying condition (i) amounts to ensuring that none of eight linear equations in \(\mathbb{Z}/m\mathbb{Z}\) are satisfied, one corresponding to each pair of \(a, b, c, a', b'\) not equal to \(\{a, b\}\) or \(\{a', b'\}\). Since the coefficient of \(c\) in each of these equations is either \(1\) or \(2\), it follows that condition (i) is satisfied for all but at most \(16\) values of \(c\). Similarly, condition (ii) is satisfied for all but \(16\) values of \(d\).

    Note that \(tb + sa = (t + s)(a + b) - (ta + sb) \equiv -1 \pmod{k}\), and so \(tb + sa \not\equiv 1 \pmod{m}\), where we have used the fact that \(k \mid m\) and \(k \geq 3\). So satisfying condition (iii) amounts to ensuring that none of sixteen linear equations in \(\mathbb{Z}/m\mathbb{Z}\) are satisfied, two corresponding to each pair \(\{x, y\} \in \{a, b, c, a', b'\}^{(2)} \setminus \{\{a', b'\}, \{a, b\}\}\). In these equations, the coefficient of $c$ is one of \(t\), \(s\) or \(t - s\). Since \(\text{gcd}(t, k) =\text{gcd}(s, k) = 1\) (because \(ta + bs \equiv 1 \pmod{k}\)), and \(k = t + s\), we have \(\text{gcd}(t - s, k) \leq 2\). By the choice of \(m\), we have \(\text{gcd}(t, m), \text{ gcd}(s, m), \text{ gcd}(t - s, m) \leq 2\). Hence each of the sixteen equations has at most \(2\) solutions. So condition (iii) is satisfied for all but at most \(32\) values of \(c\). Similarly, condition (iv) is satisfied for all but at most \(32\) values of \(d\).

    Since \(m \geq 49\), we can choose \(c\) such that conditions (i) and (iii) are satisfied. In order to satisfy condition (v), we must choose \(d\) such that none of three linear equations in \(\mathbb{Z}/m\mathbb{Z}\) are satisfied. In each of these congruences, the coefficient of $d$ is \(1\); hence condition (v) is satisfied for all but at most \(3\) values of \(d\) (for our particular choice of \(c\)).

    Since \(m \geq 52\), we can choose \(d\) such that conditions (ii), (iv) and (v) are satisfied as well.
    \end{proof}

    Let \(G_{m, N}^t\) be the \(k\)-graph obtained from \(G_{m, N}\) by, for each \(a, b \in \mathbb{Z}/m\mathbb{Z}\) with \(ta + sb = 1\), removing all edges of the form \(a^t b^s\), and adding the edges of the form \(a^{t - 1} b^s c(a, b)\) and \(a^t b^{s - 1} d(a, b)\). We call an edge of \(G_{m, N}^t\) an \textit{old} edge if it is an edge in \(G_{m, N}\); otherwise we call it a \textit{new} edge. Note that an edge \(v_1 \ldots v_k \in G_{m, N}^{t}\) is an old edge if and only if \(f(v_1) + \ldots + f(v_k) = 1\).
    
    Since \(\delta_{k - 1}(G_{m, N}^t) \geq N - k + 1\), if we can show that \(C_\ell^{(k)-}\) is not a subgraph of \(G_{m, N}^t\) for all \(N\), then it will follow that \(\pi_\text{co}(C_\ell^{(k)-}) \geq 1/m > 0\). 

   Suppose for sake of contradiction that \(G_{m, N}^t\) contains a copy of \(C_\ell^{(k)-}\). By Claim \ref{claim:tightCycle1} and the definition of \(G_{m, N}^t\), any copy of \(C_\ell^{(k)-}\) in \(G_{m, N}^t\) must contain a new edge. By the symmetry of \(t\) and \(s\), let us assume that a copy of \(C_\ell^{(k)-}\) in \(G_{m, N}^t\) contains a new edge of the form \(a^{t - 1} b^s c(a, b)\), where \(ta + sb = 1\). The following claim shows that in any such copy of \(C_\ell^{(k)-}\), the edges have one of three forms.
   
   From now on, \(a\) and \(b\) are fixed, and we write \(c\), \(a'\), \(b'\) and \(d\) to mean \(c(a, b)\), \(a'(a, b)\), \(b'(a, b)\) and \(d(a, b)\), respectively.

    \begin{claim} \label{claim:tightCycle3}
    Suppose \(v_1, \ldots, v_q\) are vertices in \(G_{m, N}^t\) and \(v_i v_{i +1} \ldots v_{i + k - 1}\) is an edge in \(G_{m, N}^t\) for each \(1 \leq i \leq q - k + 1\). Suppose 
    also that the edge \(v_1 \ldots v_k\) is of form \(a^{t - 1} b^s c\). Then for any \(1 \leq i \leq q - k + 1\),
    the edge \(v_i v_{i +1} \ldots v_{i + k - 1}\) either has the form (I) \(a^{t - 1} b^s c\), (II) \(a^{t - 2} b^s c a'\) or (III) \(a^{t - 1} b^{s - 1} c b'\).
    \end{claim}
    
    \begin{proof}
         By induction on \(i\). The base case \(i = 1\) is given, so assume \(i > 1\). By the induction hypothesis, the edge
    \(v_{i - 1} \ldots v_{i + k - 2}\) is an edge of the form (I), (II) or (III). 
    \case{1}{\(v_{i} \ldots v_{i + k -1}\) is an old edge: } If \(v_{i - 1} \ldots v_{i + k - 2}\) is of the form (II) or (III), then it is an old edge, since \(f(v_{i-1}) + \ldots + f(v_{i + k -2}) = 1\). Hence 
    any old edge containing \(v_i, \ldots, v_{i + k - 2}\) is again of the form (II) or (III), respectively, as the form of an old edge is uniquely determined by the form of any \((k - 1)\)-subset of it. 

    If \(v_{i - 1} \ldots v_{i + k - 2}\) is of the form (I), then \(\{v_i, \ldots, v_{i + k -2}\}\) is of the form \(a^{t - 1} b^s\), \(a^{t - 2} b^s c\) or \(a^{t - 1} b^{s - 1} c\). There is no old edge 
    containing \(v_i, \ldots, v_{i + k - 2}\) in the first case; and in the second and third cases any old edge containing \(v_i, \ldots v_{i + k - 2}\) is of the form (II) or (III), respectively.  
    
    \case{2}{\(v_i \ldots v_{i + k - 1}\) is a new edge: } In this case, there exist \(x, y \in \mathbb{Z}/m\mathbb{Z}\) such that the edge \(v_i \ldots v_{i + k -1}\) has one of the two forms \(x^{t - 1} y^s z\) or \(x^t y^{s-1} z\), and \(tx + sy = 1\).

    We will firstly show that \(x = a\) and \(y = b\).
    
    To this end, firstly suppose that there are indices \(i \leq j_1, j_2 \leq i + k - 2\) such that \(f(v_{j_1}) = x\), \(f(v_{j_2}) = y\).
    Since \(\{f(v_{j _1}), f(v_{j_2})\} \in \{a, b, c, a', b'\}^{(2)} \setminus \{\{a', b'\}\}\), condition (iii) in
    the definition of \(c\) implies that \( x = \)\(f(v_{j_1}) = a\), \( y = \)\(f(v_{j_2}) = b\). 
    
    Now suppose that there is no index \(i \leq j_1 \leq i + k - 2\) such that \(f(v_{j_1}) = x\). Then \(t = 2\), and exactly \(k - 2\) of \(f(v_i), \ldots, f(v_{i + k - 2})\) are equal to \(y\). Note that \(k \geq 5\) (so that \(k - 2 \geq 3\)), since \(\text{gcd}(\ell, k) = 1\) and \(\ell \nequiv 0, \pm 1 \pmod{k}\). In any of the forms (I), (II) or (III), the only element of \(\mathbb{Z}/m\mathbb{Z}\) appearing with multiplicity at least \(3\) when \(t = 2\) is \(b\). Hence \(y = b\) (and so \(x = a\)).

    Finally, if there is no index \(i \leq j_2 \leq i + k - 2\) such that \(f(v_{j_2}) = y\), then \(s = 2\) and exactly \(k - 2 \geq 3\) of \(f(v_i), \ldots, f(v_{i+k-2})\) are equal to \(x\). When \(s = 2\), the only element of \(\mathbb{Z}/m\mathbb{Z}\) appearing with multiplicity at least \(3\) in any of the forms (I), (II) or (III) is \(a\). It follows that \(x = a\) and \(y = b\).

    Hence in any case we have \(x = a\) and \(y = b\), so \(v_i \ldots v_{i+k-1}\) is either of the form (I) or of the form \(a^t b^{s - 1} d\).
    We will now show that the latter is impossible.
    
    Note that no \((k-1)\)-subset of any edge of the form (I), (II) or (III) has either of the forms \(a^t b^{s - 2} d\) or \(a^t b^{s - 1}\), since \(a\) occurs at most \(t - 1\) times in the forms (I), (II) and (III) (by condition (i) of the definition of \(c\)). Also, no \((k-1)\)-subset of any edge of the form (I), (II) or (III) has the form \(a^{t - 1} b^{s - 1} d\), since \(d \not\in \{b, c, b'\}\) by conditions (ii) and (v) of the definition of \(c\) and \(d\). Hence no new edge containing \(v_i, \ldots, v_{i + k -2}\) is of the form \(a^t b^{s - 1} d\).

    Since \(v_i \ldots v_{i + k -1}\) is either a new edge or an old edge, it follows from the above case analysis that \(v_i \ldots v_{i + k -1}\) is of the form (I), (II) or (III).
    \end{proof}

    \renewcommand{\qedsymbol}{\ensuremath{\square}}

    Now, we define the function \(\mathfrak{f} : V(G_{m, N}^t) \to \{1, 2, 3, 4\}\) such that 

        \begin{equation*}
        \mathfrak{f}(v_i) = \begin{cases*} 1 & if \(f(v_i) \in \{a, a'\}\) \\
            2 & if \(f(v_i) \in \{b, b'\}\) \\
            3 & if \(f(v_i) = c\) \\
            4 & otherwise.
        \end{cases*}
    \end{equation*}
    \(\mathfrak{f}\) is well defined by condition (i) of the definition of \(c\). It follows from Claim \ref{claim:tightCycle3} that in any copy of \(C_\ell^{(k)-}\) in \(G_{m, N}^t\) on vertices \(v_1, \ldots, v_\ell\) (in that order, with missing edge \(v_\ell v_1 \ldots v_{k-1}\)) which contains an edge of the form \(a^{t - 1} b^s c\), we have \(\mathfrak{f}(v_i) = \mathfrak{f}(v_{i + k})\) for 
    \(i \nequiv  0, -1 \pmod{\ell}\), where the indices are modulo \(\ell\). Following the proof of Claim \ref{claim:tightCycle1}, this implies that there are \(x, y \in \{1, 2, 3, 4\}\) such that among \(\mathfrak{f}(v_1), \ldots, \mathfrak{f}(v_k)\) the value \(x\) appears exactly \(t\) times, and \(y\) appears exactly \(s\) times.

    But this is a contradiction as this is not true for any of the three types of edges in Claim~\ref{claim:tightCycle3} (since \(t, s \geq 2\)). Thus \(G_{m, N}^{t}\) contains no copy of \(C_\ell^{(k)-}\), and hence \(\pi_\text{co}(C_\ell^{(k)-}) > 0\).
\end{proof}

\section{The Codegree Density of \texorpdfstring{\(Z_\ell^{(k)-}\)}{}}

\subsection{Preliminaries}

In order to prove Theorem \ref{thm:zycle} we introduce the following 
notion of a good set, which will be useful to us since they behave nicely under taking intersections.

\begin{definition} \label{def:goodSet}
    Let \(V\) be a set, and let \(\mu, \eta > 0\) be real numbers. A subset \(A \subset V^{(k-1)}\) is \textit{\((\mu, \eta)\)-good} if there exists \(t \in \mathbb{N}\) and pairwise disjoint subsets
    \(A_1, \ldots, A_t\) of \(V\) such that:
    \begin{enumerate}[(i)]
        \item \(A \supset \dot\bigcup_{i = 1}^{t} A_i^{(k - 1)}\),
        \item \(|A_i| \geq \mu |V|\) for each \(1 \leq i \leq t\),
        \item \(|V \setminus \dot\bigcup_{i=1}^{t} A_i| \leq \eta |V|\).
    \end{enumerate}
    In this case, we say any such collection of subsets \(A_1, \ldots, A_t\) \textit{corresponds} to \(A\). 
\end{definition}

The following lemma shows that the intersection of two good sets is again good, but with slightly worse parameters. This lemma allows us to prove Corollary \ref{cor:intersectionOfGoodSets}: that a finite intersection of good sets is nonempty, provided that their \(\eta\) parameters are small enough and that \(|V|\) is sufficiently large.

\begin{lemma} \label{lem:twoGoodSets}
    Suppose \(\mu, \eta > 0\), and write \(\eta' = 3\eta\). 
    Then there exists \(\mu' > 0\) such that if \(A\) and \(B\) are \((\mu, \eta\))-good subsets of \(V^{(k-1)}\), then
     \(A \cap B\) is a \((\mu', \eta')\)-good subset of \(V^{(k-1)}\).
\end{lemma}
\begin{proof}
    Let \(A_1, \ldots, A_t\) and \(B_1, \ldots, B_s\) be collections of subsets of \(V\)
    corresponding to \(A\) and \(B\), respectively, and put \(\mu' = \eta \mu^2\).
    Write
    \begin{equation*}
        I = \{(i, j) \in [t] \times [s] : |A_i \cap B_j| \geq \mu' |V|\},
    \end{equation*}
    Note that we have the following equality of sets
    \begin{equation*}
        \bigl(\dot\bigcup_{(i, j) \in I} (A_i \cap B_j)\bigr) = \bigl(\dot\bigcup_{i = 1}^{t} A_i\bigr) \cap \bigl(\dot\bigcup_{j = 1}^{s} B_j\bigr) \setminus \bigl(\dot\bigcup_{(i, j) \in ([t] \times [s])\setminus I} (A_i \cap B_j)\bigr),
    \end{equation*}
    and so
    \begin{equation*}
        \bigl(V \setminus \dot\bigcup_{(i, j) \in I} (A_i \cap B_j)\bigr) = \bigl(V \setminus \dot\bigcup_{i=1}^{t} A_i\bigr) \cup \bigl( V \setminus \dot\bigcup_{j=1}^{s} B_j\bigr) \cup \bigl( \dot\bigcup_{(i,j) \in ([t] \times [s]) \setminus I} (A_i \cap B_j) \bigr).
    \end{equation*}
    
    Using this equality of sets, we have
    \begin{equation*}
        \begin{aligned}
        |V \setminus \dot\bigcup_{(i, j) \in I} (A_i \cap B_j)|  
        & \leq |V \setminus \dot\bigcup_{i=1}^{t} A_i| + |V \setminus \dot\bigcup_{j=1}^{s} B_j| + \sum_{(i,j) \in ([t] \times [s]) \setminus I} |A_i \cap B_j| \\
            & \leq \eta |V| + \eta |V| + ts\mu' |V| \\
            & \leq \eta' |V|,
        \end{aligned}
    \end{equation*}
    where the last inequality holds because \(t, s \leq \mu^{-1}\), and hence \(ts \mu' \leq \eta\).
    Since \((A_i \cap B_j)^{(k-1)} \subset A \cap B\) for every \(1 \leq i \leq t\), \(1 \leq j \leq s\), we see that \(A \cap B\) is a \((\mu', \eta')\)-good subset of \(V^{(k-1)}\) with the corresponding collection of subsets \(\{A_i \cap B_j\}_{(i, j) \in I}\).
\end{proof}

\begin{corollary} \label{cor:intersectionOfGoodSets}
    Suppose \(\mu > 0\), \(m \in \mathbb{N}\), and write \(\eta = 3^{-\lceil\log_2(m)\rceil-1}\). Then there exists an integer \(n_0\) such that if \(S_1, \ldots, S_m\) are \((\mu, \eta)\)-good subsets of \(V^{(k-1)}\) and \(|V| > n_0\), then 
    \(\bigcap_{i = 1}^{m} S_i\) is nonempty.
\end{corollary}
\begin{proof}
    It suffices to prove the case where \(m\) is a power of \(2\). We proceed by induction.
    
    The case \(m = 1\) is trivial. Now assume that \(m = 2^c\) for some integer \(c \geq 1\).
 
    By Lemma \ref{lem:twoGoodSets}, there exists \(\mu' > 0\) (which depends only on \(\mu\), \(m\)) such that \(S_{2i - 1} \cap S_{2i}\) is 
    a \((\mu', 3^{-c})\)-good subset of \(V^{(k-1)}\) for each \(1 \leq i \leq m/2\). By
    the induction hypothesis, if \(|V|\) is sufficiently large, then \(\bigcap_{i = 1}^{m} S_i = \bigcap_{i = 1}^{m/2} (S_{2i - 1} \cap S_{2i})\) 
    is nonempty as required.
\end{proof}

\subsection{Proof of Theorem \ref{thm:zycle}}

Now we are ready to prove Theorem \ref{thm:zycle}.

\begin{proof}[Proof of Theorem \ref{thm:zycle}]
    Fix \(\varepsilon > 0\), and suppose that \(n\) is sufficiently large in terms of \(\varepsilon\). Let \(G = (V, E)\) be a \(k\)-graph with \(|V| = n\) and \(\delta_{k-1}(G) \geq \varepsilon n\).
    By Lemma \ref{lem:supersaturation}, it suffices to exhibit a homomorphic copy of \(Z_\ell^{(k)-}\) in \(G\).

    \renewcommand{\qedsymbol}{\ensuremath{\blacksquare}}

    \begin{claim} \label{claim:zycle1}
        There exist integers \(m = m(\varepsilon)\), \(n_0 = n_0(\varepsilon)\) such that if \(n > n_0\) then there
        exist sets \(f_1, \ldots, f_m \in V^{(k - 2)}\) with the property that 
        \(|\bigcup_{i=1}^{m} \overleftharpoon{N}(f_i)| > \binom{n}{k-1} - \binom{\varepsilon n}{k - 1}\).
    \end{claim}
    \begin{proof}[Proof of claim]
        We will choose the sets \(f_1, \ldots, f_m\) one at a time, and we shall write \(S_a = V^{(k-1)} \setminus \bigcup_{i=1}^{a} \overleftharpoon{N}(f_i)\).
        
        Suppose we have chosen \(f_1, \ldots, f_r\) already. By double counting and using the codegree condition, we have 
        \begin{equation*}
            \mathlarger\sum_{f \in V^{(k-2)}} |\overleftharpoon{N}(f) \cap S_r| = \mathlarger\sum_{e \in S_r} \binom{|N(e)|}{k-2} 
                                                                                \geq |S_r| \; \binom{\varepsilon n}{k-2}.
        \end{equation*}
    Hence we can choose \(f_{r + 1}\) such that 
    \begin{equation*}
        |\overleftharpoon{N}(f_{r+1}) \cap S_r| \geq |S_r|\binom{\varepsilon n}{k-2} \big/ \binom{n}{k-2} > c \cdot |S_r|,
    \end{equation*}
    where \(c = \tfrac{\varepsilon^{k-2}}{2}\), and the final inequality holds if \(n\) is sufficiently large since \(\binom{\varepsilon n}{k - 2} \sim \varepsilon^{k - 2}\binom{n}{k - 2}\). With this choice of \(f_{r+1}\),
     \(|S_{r+1}| < (1 - c)|S_r|\), so by induction we obtain that 
    \(|S_a| < (1 - c)^{a} \cdot \binom{n}{k-1}\). 
    
    Choose \(m\) such that \((1 - c)^{m} < \tfrac{\varepsilon^{k-1}}{2}\). Using the fact that
    \( \binom{\varepsilon n}{k - 1} \sim \varepsilon^{k-1}\binom{n}{k-1}\), if \(n\) is sufficiently large we have 
    \begin{equation*}
        |\bigcup_{i = 1}^{m} \overleftharpoon{N}(f_i)| = \binom{n}{k - 1} - |S_m| > \binom{n}{k - 1} - \frac{\varepsilon^{k -1}}{2} \binom{n}{k - 1} > \binom{n}{k - 1} - \binom{\varepsilon n}{k - 1}. \qedhere
    \end{equation*}
    \end{proof}
    
    Now define for each \(1 \leq i \leq m\) the following sets:
    \begin{equation*}
        \mathcal{N}_i = \bigcup_{w \in V \setminus f_i}  N(f_i w)^{(k - 1)}.
    \end{equation*}
    Thus if \(e \in \mathcal{N}_i\), then \(f_i w \rhd e\) for some \(w \in V \setminus f_i\). The following claim, together with Corollary~\ref{cor:intersectionOfGoodSets}, shows that \(\bigcap_{i = 1}^{m}\mathcal{N}_i\) is nonempty. 
    
    \begin{claim} \label{claim:zycle2}
        For every \(\eta > 0\), there exist \(\mu = \mu(\eta, \varepsilon) > 0\) and an integer \(n_0 = n_0(\eta, \varepsilon)\) such that if \(n > n_0\) then
        \(\mathcal{N}_i\) is a \((\mu, \eta)\)-good subset of \(V^{(k - 1)}\) for each \(1 \leq i \leq m\).
    \end{claim}
    
    \begin{proof}[Proof of claim] 
        Fix \(1 \leq i \leq m\). We will choose vertices \(w_1, \ldots, w_q\) one
        at a time, and we will write \(S_a = V \setminus \bigcup_{j=1}^{a} N(f_i w_j)\). We will stop choosing vertices if it is ever the case that \(S_q = f_i\).
        
        Suppose we have chosen \(w_1, \ldots, w_r\) already. By double counting and using the codegree condition, we have 
        \begin{equation*}
            \sum_{w \in V \setminus f_i}  |N(f_i w) \cap S_r| = \sum_{x \in S_r \setminus f_i} |N(f_i x)| \geq |S_r \setminus f_i| \cdot {\varepsilon n} \geq |S_r| \cdot c n.
        \end{equation*}
        where \(c = \varepsilon / k\), and the final inequality holds since \(S_r \neq f_i\). Hence we can choose \(w_{r+1}\) such that 
        \begin{equation*}
            |N(f_i w_{r+1}) \cap S_r| \geq \frac{|S_r| \cdot c n}{n - k+2} > c \cdot |S_r|,
        \end{equation*}
        With this choice of \(w_{r+1}\), \(|S_{r+1}| < (1 - c)|S_r|\),
        and so by induction \(|S_a| < (1 - c)^a \cdot n\).
    
        Choose \(q\) such that \((1 - c)^{q} < \eta\). Then \(|S_q| < \eta n\), so there exists a least 
        \(1 \leq p \leq q\) such that \(|S_p| < \eta n\). 

        Write \(T_1 = V \setminus S_1\), and \(T_a = S_{a - 1} \setminus S_a\) for \(2 \leq a \leq p\).
        Note that for each \(1 \leq r \leq p\), \(T_r^{(k-1)} \subset N(f_i w_r)^{(k - 1)} \subset \mathcal{N}_i\), and 
        \(|T_r| > c\cdot |S_{r-1}| \geq c\eta n\). Finally, \(|V \setminus \bigcup_{j=1}^{p} T_j| = |S_p| < \eta n\) by the definition of \(p\). 
        Hence \(\mathcal{N}_i\) is a 
        \((c\eta, \eta)\)-good subset of \(V^{(k-1)}\), with corresponding sets \(T_1, \ldots, T_p\).
    \end{proof}

    \renewcommand{\qedsymbol}{\ensuremath{\square}}

    By Claim \ref{claim:zycle2} and Corollary \ref{cor:intersectionOfGoodSets}, it follows that \(\bigcap_{i = 1}^{m} \mathcal{N}_i\) is nonempty - 
    so we can choose some \(e_1 \in \bigcap_{i = 1}^{m} \mathcal{N}_i\) arbitrarily. 

    Now, by the codegree condition we can choose  \(e_2, \ldots, e_{\ell-2} \in V^{(k-1)}\) such that 
    \begin{equation*}
        e_1 \rhd e_2 \rhd \cdots \rhd e_{\ell-2}.
    \end{equation*}

    Again by the codegree condition we have \(|N(e_{\ell-2})^{(k - 1)}| \geq \binom{\varepsilon n}{k-1}\), 
    and so the set
    \begin{equation*}
        S = \biggl(\bigcup_{i=1}^{m} \overleftharpoon{N}(f_i)\biggr) \cap N(e_{\ell-2})^{(k - 1)}
    \end{equation*}
    is nonempty - so we can choose \(e_{\ell-1} \in S\) arbitrarily. Then there exists \(1 \leq i \leq m\) 
    such that \(f_i \in N(e_{\ell-1})^{(k - 2)}\).
    
    Finally, since \(e_1 \in \mathcal{N}_i\), there exists a vertex \(w\) such that 
    \(f_i w \rhd e_1\). 

    Summarising the above, we see that \(G \cup e_{\ell-1}w\) contains the homomorphic copy of \(Z_\ell^{(k)}\) given by
    \begin{equation*}
        e_1 \rhd \cdots \rhd e_{\ell-1} \rhd f_i w \rhd e_1.
    \end{equation*}
    That is, \(G\) contains a homomorphic copy of \(Z_\ell^{(k)-}\) as required.
\end{proof}

\section*{Acknowledgements}
The author is extremely grateful to Oliver Janzer for supervising this research and for all his comments on the manuscript.

\bibliographystyle{abbrv}
\bibliography{references}

\end{document}